\pdfoutput=1
\documentclass[english]{ourlematema}
\usepackage{amsmath, amsthm, amssymb, hyperref, color,mathtools}
\usepackage{enumerate}
\usepackage{MnSymbol} 
\usepackage{graphicx}
\usepackage{caption}
\usepackage[all]{xypic}
\usepackage{verbatim}
\usepackage{chemfig, chemnum}
\usepackage{tikz}
\usepackage[linesnumbered,lined,commentsnumbered]{algorithm2e}
\usepackage{caption}
\usepackage{subcaption}
\usepackage{float}
\usepackage{makecell}
\usepackage{array}
\usepackage{fancyvrb}

\usepackage{tikz,pgfplots,tikz-3dplot}
\usepackage{cleveref}
\pgfplotsset{compat=1.18}
\DeclareMathOperator{\Ann}{Ann}

\DeclareMathOperator{\dlog}{dlog}
\renewcommand{\d}{\mathrm{d}}

\newtheorem{theorem}{Theorem}
\numberwithin{theorem}{section}
\newtheorem{proposition}[theorem]{Proposition}

\newtheorem{remark}[theorem]{Remark}
\newtheorem{example}[theorem]{Example}

\theoremstyle{definition}

\newcommand{\RR}{\mathbb{R}}

\newcommand{\CC}{\mathbb{C}}
\newcommand{\ZZ}{\mathbb{Z}}

\newcommand{\mainprop}{
Let $\ell_1,\ldots,\ell_m$ be as in~\eqref{eq:elli}, and $\phi$ the correlator function~\eqref{eq:corr}. Let $H$ be the homogeneity operator~\eqref{eq:homogeneous}, $\{L_i\}$ the operators~\eqref{eq:Li} arising from the individual hyperplanes, and $\{P_j\}$ and $\{Q_k\}$ the operators constructed from circuits and syzygies, respectively, as was explained above. Then the left $D$-ideal generated by them annihilates $\phi$, i.e.,
\begin{align*}
   \langle H,\{L_i\},\{P_j\},\{Q_k\}\rangle \, \subset \, \Ann_{D(s,\nu)}(\phi) \,.
\end{align*}
}

\title{Differential Equations for\\
{ }Moving Hyperplane Arrangements{ }}

\author{Anaëlle Pfister}
\address{%
Max Planck Institute for Mathematics in the Sciences, Leipzig, Germany\\
\email{anaelle.pfister@mis.mpg.de}
}

\author{Anna-Laura Sattelberger}
\address{%
Max Planck Institute for Mathematics in the Sciences, Leipzig, Germany\\
\email{anna-laura.sattelberger@mis.mpg.de}
}

\authormark{A.~Pfister \ - \  A.-L.~Sattelberger}

\begin{document}
\maketitle

\begin{abstract}
\noindent We investigate Mellin integrals of products of hyperplanes, raised to an individual power each. We refer to the resulting functions as {\em combinatorial correlators}. We investigate their behavior when moving the hyperplanes individually. To encode these functions as holonomic functions in the constant terms of the hyperplanes, we aim to construct a holonomic annihilating $D$-ideal purely in terms of the hyperplane arrangement. 
\end{abstract}

\section{Introduction}\label{sec:intro}
We fix $m$ linear forms 
$\ell_1(x),\ldots,\ell_m(x)$
in $n$ variables
$x = (x_1,\ldots,x_n)$. They encode a central hyperplane arrangement in~$\RR^n$.  
We introduce shift parameters $c_1,c_2,\ldots,c_m$,
and we consider the $m$ affine hyperplanes $\{ x \in \RR^n : \ell_i(x) = c_i\}$
for $i=1,\ldots,m$. We augment this by the coordinate hyperplanes
$\{x \in \RR^n: x_j = 0 \}$ for $j=1,\ldots,n$.
The complement of $\CC^n$ by this arrangement of $m+n$ hyperplanes is a very affine variety $X$ that depends on the unknowns $c_1,c_2,\ldots,c_m$.

Our object of study is the following generalized Euler integral \cite{AFST22} associated to the $m$ shifted linear forms,
\begin{align}\label{eq:corrintro}
 \phi(c_1,\ldots,c_m) \, =\, 
\int_\Gamma  
\left(\ell_1(x)-c_1\right)^{s_1}\cdots \left(\ell_m(x)-c_m\right)^{s_m} \,
x_1^{\nu_1} \cdots x_n^{\nu_n} \, 
\frac{\d x_1}{x_1} \wedge \cdots \wedge \frac{\d x_n}{x_n} \, ,
\end{align} 
where $\Gamma$ is a twisted $n$-cycle of~$X$, and $s\in (\CC\setminus \{0\})^m$ and $\nu \in (\CC\setminus \{1\} )^n$ can be complex.
This is the Mellin transform of $\prod_{i=1}^m(\ell_i-c_i)^{s_i}$, but considered as a function of $c = (c_1,\ldots,c_m)$. We refer to the function~\eqref{eq:corrintro} as a {\em combinatorial correlator}.  
Our choice of name is a reference to the theory of {cosmological correlators}, and in particular to the recent article~\cite{DEcosmological}, in which the authors study the integral $\phi(c)$ in the special
case when the linear forms $\ell_i$ range over subsums of the coordinates~$x_j$, and $\nu_1 = \cdots = \nu_n = \epsilon$. In a cosmological setup, this function measures quantities such as the strength of correlations in the first light released in the hot big bang. The differential as well as difference equations behind cosmological correlator functions are tackled from an algebraic perspective in~\cite{CosmoDmod}.

We here seek to determine differential equations that annihilate $\phi(c)$
for all twisted cycles~$\Gamma $. These equations
correspond to a left ideal $I\subset D$ in the \mbox{$m$-th} Weyl algebra in the $c$-variables. More precisely, we aim to represent $\phi$ as a holonomic function. It is well-known that~\eqref{eq:corrintro} is the solution to a restricted GKZ system~\cite{GKZ90}; but these are difficult to compute in practice. We here offer a direct, combinatorial approach, employing the hyperplane~arrangement~only. 

Our construction is purely combinatorial and depends on the hyperplane arrangement only---however, not only on the matroid of the arrangement, as shown in \Cref{sec:U23}. Our main result, summarized in \Cref{prop:ann}, is the construction of an annihilating $D$-ideal of the correlator~$\phi$. It reads~as~follows.
\begin{thm*}
Let $\ell_1,\ldots,\ell_m$ be as in~\eqref{eq:elli}, and $\phi$ the correlator function~\eqref{eq:corr}. Let $H$ be the homogeneity operator~\eqref{eq:homogeneous}, $\{L_i\}$ the operators~\eqref{eq:Li} arising from the individual hyperplanes, and $\{P_j\}$ and $\{Q_k\}$ the operators constructed from circuits and syzygies, respectively, as was explained above. Then the left $D$-ideal generated by them annihilates $\phi$, i.e.,$ \langle H,\{L_i\},\{P_j\},\{Q_k\}\rangle \, \subset \, \Ann_{D(s,\nu)}(\phi) \,.$
\end{thm*}
\Cref{sec:examples} showcases that indeed, in several examples, the holonomic rank of our $D$-ideal attains the upper bound for the holonomic rank of the full annihilating $D$-ideal of $\phi$. In particular, it encodes $\phi$ as a holonomic function. Our study also suggests a relation of the singular locus of $I$ to the discriminantal arrangement of the hyperplane arrangement. In \Cref{prop:singlocn2}, we prove that, for line arrangements, the singular locus of our $D$-ideal is contained in the discriminantal arrangement. While working on this article, the work \cite{FevolaHeo24} 
of Fevola and Matsubara-Heo 
on Euler discriminants of complements of hyperplanes appeared. Their results also recover the singularities of generalized~Euler~integrals.

\smallskip 
In short, we give a combinatorial construction of an annihilating $D$-ideal of~$\phi$~\eqref{eq:corrintro}.  
For software, we use the {\tt Dmodules} package~\cite{DmodM2} in {\em Macaulay2}~\cite{M2}, the package {\tt HolonomicFunctions}~\cite{HolFun} in Mathematica, and the $D$-module libraries~\cite{ABLMS} in {\sc Singular:Plural}~\cite{Singular,Plural}.
We provide our code via GitLab at 
\linebreak \href{https://uva-hva.gitlab.host/universeplus/differential-equations-for-moving-hyperplane-arrangements.git}{https://uva-hva.gitlab.host/universeplus}.
We surmise that the methods developed here will ultimately be useful for cosmology~and~particle~physics. 

\bigskip

\noindent {\bf Outline.} \Cref{sec:background} recalls background on the mathematical tools that we 
employ. In \Cref{sec:construction}, we construct an annihilating $D$-ideal of the correlator function~\eqref{eq:corrintro} purely from the hyperplane arrangement. In \Cref{sec:examples}, we showcase our methods with examples. \Cref{sec:outlook} gives an outlook to future~work.

\section{Preliminaries}\label{sec:background}
We here recall mathematical tools needed for our study. They reach from operator algebras through twisted cohomology to discriminantal~arrangements.
\subsection{Operator algebras}
\paragraph{Differential operators}
The operators we seek for are elements of the $m$-th Weyl algebra in the \mbox{$c$-variables}, denoted $D_m$ or just $D$, 
\begin{align*}
    D_m \,=\, \CC [c_1,\ldots,c_m]\langle \partial_{c_1},\ldots,\partial_{c_m} \rangle,
\end{align*}
where $\partial_{c_i}=\frac{\partial}{\partial c_i}$ is the partial derivative with respect to~$c_i$. It is obtained from the free $\CC$-algebra generated by $c_1,\ldots,c_m, \partial_{c_1},\ldots,\partial_{c_m} $, modulo the following relations. All generators are assumed to commute, except $c_i$ and $\partial_{c_i}$: they obey Leibniz' rule, i.e., $\partial_{c_i} c_i-c_i\partial_{c_i} = 1$ for $i=1,\ldots,m$. Systems of linear PDEs are encoded as left ideals $I\subset D_m$ in the Weyl algebra. 

The {\em singular locus} $\operatorname{Sing}(I)\subset \CC^m$ of a $D_m$-ideal $I$  is derived from the initial ideal of $I$ with respect to the weight vector $(0,1)\in \RR^{2m}$. It encodes where holomorphic solutions to the system of PDEs encoded by $I$ might have singularities; we refer to \cite[Definition~1.12]{SatStu19} for the precise construction.
We will also need the {\em rational Weyl algebra}, denoted $R_m=\CC(c_1,\ldots,c_m)\langle \partial_{c_1},\ldots,\partial_{c_m}\rangle$, for instance to define the {\em holonomic rank} of a $D_m$-ideal~$I$, which is the dimension of $R_m/R_mI$ as a $\CC(c_1,\ldots,c_m)$-vector space.
We will denote the action of operators on a function $f(c_1,\ldots,c_m)$ by a bullet; e.g., $\partial_{c_i}\bullet f=\frac{\partial f}{\partial c_i}$. 
The {\em annihilator} of a function $f(c_1,\ldots,c_m)$, denoted $\Ann_D(f)\coloneqq\{ P\in D_m | P\bullet f=0 \}$, is the $D_m$-ideal consisting of \underline{all} $P\in D_m$ that annihilate~$f$. We point out that, in order to encode $f$ as a holonomic function, it is sufficient to construct a subideal $ I\subset \Ann_D(f)$ such that $I$ has finite holonomic rank, and this is what we tackle in this article.
Instead of~$\CC$, we will also use the field $\CC(s,\nu)=\CC(s_1,\ldots,s_m,\nu_1,\ldots,\nu_n)$ for the field of coefficients, and sometimes denote the resulting Weyl algebra by $D(s,\nu)$.

\paragraph{Shift operators} In our construction of annihilating differential operators, we are also going to utilize {shift operators}. Differential operators encode linear PDEs; shift operators encode recurrence relations. Denoting the discrete shift of the variable $\nu_i$ by $\pm 1$ by $\sigma_{\nu_i}^{\pm 1}: \nu_i\mapsto \nu_i\pm 1$, they obey 
$\sigma_{\nu_i}^{\pm 1}\nu_i = \left(\nu_i\pm 1\right)\sigma_{\nu_i}^{\pm 1}$ for $i=1,\ldots,n$.
Such operators are encoded as elements of the {\em shift algebra}, denoted
\begin{align*}
    \mathcal{S}_n \,=\, \CC[\nu_1,\ldots,\nu_n]\langle \sigma_{\nu_i}^{\pm 1},\ldots,\sigma_{\nu_n}^{\pm 1} \rangle \, .
\end{align*}
In our study, we are going to construct recurrence relations for $\phi$, both in the\linebreak$\nu$- and $s$-variables, from which we will derive elements in $\Ann_{D(s,\nu)}(\phi)$.

\subsection{Twisted cohomology}\label{sec:twistedcoh}
Let $f_1,\ldots,f_m \in \CC[x_1^{\pm 1},\ldots,x_n^{\pm 1}]$ be Laurent polynomials, and denote by $f=f_1\cdots f_m$ their product. Its complement $X=(\mathbb{G}_m^n\setminus V(f))$ in the algebraic $n$-torus $\mathbb{G}_m^n=\operatorname{Spec}(\CC[x_1^{\pm 1},\ldots,x_n^{\pm 1}])$  is a very affine variety via the graph embedding. In slight abuse of notation, we denote $\mathbb{G}_m^n$ by its closed points, $(\CC^{\ast})^n$.

We are going to consider the complex of algebraic differential forms on~$X$, and twist the differential by the logarithmic form 
\begin{align*}
\omega \,=\, 
\dlog \big( x_1^{\nu_1}\cdots x_n^{\nu_n} \cdot\prod_{j=1}^m f_j^{s_j}\big) \, ,
\end{align*}
i.e., our differential is
\begin{align*}
    \nabla_\omega \,=\, \d + \big( \sum_{j=1}^m s_j \frac{\d f_j}{f_j} \, + \, \sum_{i=1}^n \nu_i \frac{\d x_i}{x_i} \big) \wedge \  \ ,
\end{align*}
with $(s,\nu)\in \CC^{m+n}$, and $\d$ denotes the total differential. The $k$-th cohomology group of this complex is the {\em $k$-th twisted cohomology group} of $X$ and is denoted by~$H^k(X,\omega)$. It is generated by the forms 
\begin{align}\label{eq:genkforms}
x^af^b \, \frac{\d x_{i_1}}{x_{i_1}} \wedge \cdots \wedge \frac{\d x_{i_l}}{x_{i_l}} \, ,
\end{align}
where $(a,b)\in \ZZ^{n+m}$, and $1\leq l \leq k$. Any $(n-1)$-form $\phi\in \Omega^{n-1}(X)$ gives rise to a shift relation among integrals $\int f^{s+b}x^{\nu+a} \frac{\d x}{x}$ by expressing $\nabla_\omega(\phi)$ in terms of the generators in~\eqref{eq:genkforms}. Relations obtained like this are called ``IBP relations,'' see e.g.~\cite{AFST22} for more details.

\subsection{Relative twisted cohomology}
We return to our setup of $m$ hyperplanes in $n$-space as in the Introduction. Let
\begin{align*}
    X_c  \,=\, \big\{\left(c_1,\ldots,c_m,x_1,\ldots,x_n\right) \, |\  x_1\cdots x_n \cdot \prod_{i=1}^m \left( \ell_i -c_i \right) \neq 0 \big\}  \, \subset \, \CC^m \times \CC^n
\end{align*}
and $\nabla_{\omega} = \d_{x} + \dlog_{x}(x^{\nu}\cdot \prod_{i=1}^m(\ell_i-c_i)^{s_i})$, i.e., 
\begin{align*}
\nabla_\omega \,=\, \d_x + \big( \sum_{j=1}^m s_j \frac{\d_x \ell_j}{\ell_j-c_j} \, + \, \sum_{i=1}^n \nu_i \frac{\d x_i}{x_i} \big) \wedge \ \ .
\end{align*}
As is indicated by the subscript, the differential is taken only w.r.t.\ the $x$-variables and not the $c$-variables. We will need to consider relative differential $k$-forms,
\begin{align*}
\Omega_{X_c/\CC^m}^k \,=\, \bigoplus_{i_1<\cdots<i_k}  \mathcal{O}_{X_c} \d x_{i_1}\wedge \cdots \wedge \d x_{i_k}\, ,
\end{align*}
and will mean its global sections throughout.
The twisted relative cohomology 
\begin{align*}\begin{split}
H^n\left(X_c/\CC^m,\omega\right) \otimes_{\CC} \CC(s,\nu) &\,=\, \Omega^n_{X_c/\CC^m}(s,\nu) / \nabla_\omega(\Omega^{n-1}_{X_c/\CC^m})(s,\nu) 
\\ & \,=\, D_m(s,\nu) \cdot \left[  \frac{\d x_1}{x_1}\wedge \cdots \wedge \frac{\d x_n}{x_n} \right] 
\end{split}\end{align*}
is a holonomic $D_m(s,\nu)$-module, with the action of $\partial_{c_i}$ being
\begin{align*}
    \partial_{c_i}\bullet \left[ f(c,x)\cdot \d {x_1}\wedge \cdots \wedge \d x_n \right] \,=\, \left[ \left( \frac{\partial f}{\partial c_i} - \frac{s_i}{\ell_i-c_i }\cdot f \right) \, \d x_1 \wedge \cdots \wedge \d x_n \right] \,.
\end{align*}
For $c$ generic, the holonomic rank of $H^n(X_c/\CC^m,\omega)$ is $|\chi(X_c)|$, the signed Euler characteristic of~$X_c$; cf.\ for instance \cite[Theorem~1.1]{AFST22} for elaborations. 
Since \begin{align*}
    D_m(s,\nu)/\Ann(\phi) \,\hookrightarrow \, D_m(s,\nu)\cdot \left[ \frac{\d x_1}{x_1}\wedge \cdots \wedge \frac{\d x_n}{x_n}  \right]\, , \ \, 1 \mapsto \left[ \frac{\d x_1}{x_1}\wedge \cdots \wedge \frac{\d x_n}{x_n}  \right]\, ,
\end{align*}
is an embedding of $D$-modules, the number of bounded regions is an upper bound for the holonomic rank of $\Ann_D(\phi)$, the full annihilating $D$-ideal of~$\phi$. We refer to \cite[Section~5]{FevolaHeo24} for more details.

\subsection{Discriminantal arrangements}\label{sec:discrarr}
We here follow~\cite{BayerBrandt,DiscrArr}. 
Let $\mathcal{A}$ be a fixed arrangement of $m$ affine hyperplanes in~$\RR^n$ that are in general position. The set of general position arrangements whose hyperplanes are parallel to those of $\mathcal{A}$, is the complement of a central arrangement, in $\RR^m$. It is the {\em discriminantal arrangement} of~$\mathcal{A}$. As pointed out in \cite[Section 2]{BayerBrandt}, the construction also works for multiarrangements, i.e., hyperplanes are allowed to occur with multiplicity greater than~$1$. For instance, for $n=1$ and $m$ points $V(x-c_i)$ on the line, $\RR$, the discriminantal arrangement is the braid arrangement 
\begin{align*}
V\big( \prod_{1\leq i<j\leq m}\left(c_i-c_j\right)\big) \, \subset \, \RR^m \, .
\end{align*}
Discriminantal arrangements will occur later on in the study of the singular locus of our combinatorially constructed annihilating $D$-ideal of~$\phi$.

\section{Construction of annihilating differential operators}\label{sec:construction}
We now explain how to construct differential operators that annihilate the correlator function. Consider $m$ hyperplanes through the origin in affine $n$-space,  \begin{align}\label{eq:elli}
\ell_i \,=\, a_1^{(i)}x_1+\cdots + a_n^{(i)}x_n\, , \quad  i\, =1,\ldots,m\, .
\end{align}
\noindent The {\em combinatorial correlator} is the function
\begin{align}\label{eq:corr}
 \phi(c) \, =\, 
\int_\Gamma  
\left(\ell_1(x)-c_1\right)^{s_1}\cdots \left(\ell_m(x)-c_m\right)^{s_m} \,
x_1^{\nu_1} \cdots x_n^{\nu_n} \, 
\frac{\d x_1}{x_1} \wedge \cdots \wedge \frac{\d x_n}{x_n} \, ,
\end{align}
in $c=(c_1,\ldots,c_m)$,
where $\Gamma$ is any twisted $n$-cycle. To be precise, 
\begin{align*}
\Gamma \,\in\, H_n \big(\big(\CC^{\ast}\big)^n \,\backslash \, V\big( \prod_{i=1}^m \left(\ell_i-c_i\right)\big),\mathcal{L}_\omega\big) \, ,
\end{align*}
i.e., $\Gamma$ is an element of
the $n$-th homology with coefficients in the local system~$\mathcal{L}_{\omega}$ of flat sections of the connection $\nabla_{-\omega}=\d - \omega \wedge \ $; see~\cite{AFST22} for~details.

Since the correlator function is homogeneous of degree $\sum_{i=1}^n \nu_i +\sum_{j=1}^m s_j$, i.e., $\phi(\lambda c_1,\ldots,\lambda c_m)=\lambda^{s_1+\cdots+s_m+\nu_1+\cdots+\nu_n} \cdot \phi(c_1,\ldots,c_m)$, one derives the following lemma from Euler's homogeneous function theorem.
\begin{lemma}
The combinatorial correlator function~\eqref{eq:corr} is annihilated by
\begin{align}\label{eq:homogeneous}
    H \, \coloneqq \, c_1\partial_{c_1}+\cdots +c_m\partial_{c_m} - \big( \sum_{i=1}^n \nu_i + \sum_{j=1}^m s_j \big) \, \in \, \operatorname{Ann}_{D(s,\nu)}\left( \phi \right) \, ,
\end{align}
to which we are going to refer as the ``homogeneity operator.''
\end{lemma}

The partial derivatives of $\phi$ with respect to the $c$-variables are 
\begin{align*}
\frac{\partial \phi}{\partial {c_i}}(c_1,\ldots,c_m) \,=\, -s_i \, \cdot\int_{\Gamma}\frac{\prod_{j=1}^m \left(\ell_j-c_j\right)^{s_j}}{\ell_i-c_i}x_1^{\nu_1}\cdots x_n^{\nu_n} \frac{\d x_1}{x_1} \wedge \cdots \wedge \frac{\d x_n}{x_n} \, ,
\end{align*}
for $i=1,\ldots,m$. 
We can therefore identify the action of backwards shifts in $s_i$, $\sigma_{s_i}^{-1}:s_i\mapsto s_i-1$, with the action of differential operator $\partial_{c_i}$ on $\phi$ via
\begin{align}\label{eq:shiftsdiffsc}
    \partial_{c_i}\bullet \phi \,=\, -s_i\sigma_{s_i}^{-1}\bullet \phi  \, .
\end{align}
Note that this is not an equality of operators---it holds only when applied to $\phi$. 

Moreover, for $k \leq m$ and $i_1,\ldots,i_k$ distinct, 
\begin{align}
    \frac{\partial^k \phi}{\partial{c_{i_1}}\cdots \partial {c_{i_k}}}(c_1,\ldots,c_m)=(-1)^{k} s_{i_1} \cdots s_{i_k}\cdot \int_{\Gamma}\frac{f^s}{f_{i_1} \cdots f_{i_k}}x_1^{\nu_1}\cdots x_n^{\nu_n} \frac{\d x_1}{x_1} \wedge \cdots \wedge \frac{\d x_n}{x_n}\, ,
\end{align}
so that
\begin{align*}
{\partial_{c_{i_1}}\cdots \partial_{c_{i_k}}} \bullet \phi \,=\,(-1)^k{s_{i_1}\cdots s_{i_k}}\sigma_{s_{i_1}}^{-1}\cdots \sigma_{s_{i_k}}^{-1} \bullet \phi \, .
\end{align*}

For the derivatives with respect to the $x$-variables, one has 
\begin{align*}
\begin{split}
    \partial_{x_j} \bullet \prod_{i=1}^m (\ell_i-c_i)^{s_i} &\,=\, \sum_{i=1}^m s_ia_j^{(i)}(\ell_i-c_i)^{s_i-1}\prod_{k\neq i} (\ell_k-c_k)^{s_k},
\end{split}
\end{align*}
and hence
\begin{align*}
 -\sum_{i=1}^m a_j^{(i)}\partial_{c_i} \bullet   \prod_{i=1}^m (\ell_i-c_i)^{s_i} \,=\, \partial_{x_j} \bullet \prod_{i=1}^m (\ell_i-c_i)^{s_i} \,.
\end{align*}
We exploit this for carrying out an integration by parts:
\begin{align}\label{eq:IBP}
\begin{split}
   & -\sum_{i=1}^m a_j^{(i)} \partial_{c_i} \bullet \phi  \,=\, \int_\Gamma \left( \partial_{x_j} \bullet \prod_{i=1}^m (\ell_i-c_i)^{s_i} \right)  x_1^{\nu_1}\cdots x_n^{\nu_n} \,\frac{\d x_1}{x_1}\wedge \cdots  \wedge \frac{\d x_n}{x_n} \quad  \\
     & \quad \ \ \ \stackrel{\text{IBP}}{=}\, -(\nu_j-1) \int_\Gamma \prod_{i=1}^m (\ell_i-c_i)^{s_i}   x_1^{\nu_1}\cdots x_j^{\nu_j-1}\cdots \, x_n^{\nu_n} \,\frac{\d x_1}{x_1}\wedge \cdots  \wedge \frac{\d x_n}{x_n} \, . 
     \end{split}
\end{align}
For the second equality in~\eqref{eq:IBP}, we use integration by parts. 
Since the integration cycle does not have a boundary, the term $[\prod_{i=1}^m (\ell_i-c_i)^{s_i}   x_1^{\nu_1}\cdots x_j^{\nu_j-1}\cdots \, x_n^{\nu_n}]_{|\partial \Gamma}$ in the IBP formula vanishes.
We hence identify the action of inverse shifts in the $\nu$'s by first-order differential operators in the $c$'s as
\begin{align}\label{eq:shiftnu}
   (\nu_j-1) \sigma_{\nu_j}^{-1}\bullet \phi \,=\, {\sum_{i=1}^m a_j^{(i)} \partial_{c_i}} \bullet \phi  \, .
\end{align}
We point out the recursive nature of this ``replacement rule.'' For the action of~$\nu_j^2$ on $\phi$, for instance, this implies 
\begin{align*}
\sigma_{\nu_j^2}^{-1}\bullet \phi \,=\, \frac{1}{(\nu_j-1)(\nu_j-2)}\cdot \sum_{i,k=1}^m {a_j^{(i)}a_{j}^{(k)}} \partial_{c_i}\partial_{c_k} \bullet \phi \, .
\end{align*}
In the following two subsections, we are going to exploit \eqref{eq:shiftsdiffsc} and \eqref{eq:shiftnu} to construct elements of $\Ann_{D(s,\nu)}(\phi)$ from recurrence relations for $\phi$ in the $s$ and $\nu$'s.

\subsection{Differential operators from individual hyperplanes}\label{sec:diffIBP}
Consider hyperplanes  $\ell_i(x)=a_1^{(i)}x_1+\cdots+a_n^{(i)}x_n$, $i=1,\ldots,m$, and $\phi$ as in~\eqref{eq:corr}. Each of the $\ell_i$'s gives rise to an operator $L_i\in \Ann_D(\phi)$, as we explain now.
Denoting the 
integrand in the correlator~\eqref{eq:corr} shorthand by $\eta=(\prod_{i=1}^m(\ell_i-c_i)^{s_i})x^\nu\frac{\d x}{x}$,
one directly reads that
\begin{align*}
    \int_\Gamma a_1^{(i)}x_1 \eta+\cdots + \int_\Gamma a_n^{(i)}x_n \eta - \int_\Gamma c_i \eta & \ =\, \int_\Gamma \left( \ell_i -c_i \right) \eta \, .
\end{align*}
Hence
 \begin{align*}
\left[\ell_i(\sigma_{\nu_1},\ldots,\sigma_{\nu_n}) -c_i \right]\bullet \phi \,=\, \sigma_{s_i} \bullet  \phi \, .
 \end{align*}
 %
%
Multiplying this identity with $\sigma_{s_i}^{-1}\sigma_{\nu_1}^{-1}\cdots \sigma_{\nu_n}^{-1}$ from the left yields the equality
\begin{align*}
\left[\sigma_{s_i}^{-1}\sigma_{\nu_1}^{-1}\cdots \sigma_{\nu_n}^{-1}\ell_i(\sigma_{\nu_1},\ldots,\sigma_{\nu_n}) -c_i \sigma_{s_i}^{-1}\sigma_{\nu_1}^{-1}\cdots \sigma_{\nu_n}^{-1}\right]\bullet \phi \,=\, \sigma_{\nu_1}^{-1}\cdots \sigma_{\nu_n}^{-1} \bullet  \phi \, .
\end{align*}
Replacing $\sigma_{s_i}^{-1}$ and $\sigma_{\nu_i}^{-1}$ as in \eqref{eq:shiftsdiffsc} and \eqref{eq:shiftnu}, yields $m$ differential operators $L_1,\ldots,L_m$ in the $c$-variables of order at most $n+1$ that annihilate $\phi$, namely
\begin{align}\begin{split}\label{eq:Li}
    L_i \,=\, -\frac{\partial_{c_i}}{s_i}\big(\sum_{j=1}^n \big({\sum_{k=1}^m a_1^{(k)} \partial_{c_k}} \cdot \ldots \cdot  {\sum_{k=1}^m a_{j-1}^{(k)} \partial_{c_k}} \cdot {\sum_{k=1}^m a_{j+1}^{(k)} \partial_{c_k}} \cdot \ldots \cdot  {\sum_{k=1}^m a_n^{(k)} \partial_{c_k}}\big) a_j^{(i)}\big) \\
   + \, c_i\frac{\partial_{c_i}}{s_i}\big({\sum_{k=1}^m a_1^{(k)} \partial_{c_k}} \cdot \ldots \cdot  {\sum_{k=1}^m a_n^{(k)} \partial_{c_k}} \big)
    \,-\, \big({\sum_{k=1}^m a_1^{(k)} \partial_{c_k}} \cdot \ldots \cdot  {\sum_{k=1}^m a_n^{(k)} \partial_{c_k}}\big)\, .
\end{split}\end{align}

\begin{remark}
We point out that IBP relations derived from the twisted differentials of $(n-1)$-forms, as explained in \Cref{sec:twistedcoh}, give rise to the trivial differential operator in $\Ann_D(\phi)$ only, when replacing via \eqref{eq:shiftsdiffsc} and \eqref{eq:shiftnu}.
\end{remark}

We summarize the findings of this subsection in
\begin{lemma}
Denote by $L_i$, $i=1,\ldots,m$, the differential operators~\eqref{eq:Li} derived from the shifted hyperplanes $\{ \ell_i-c_i=0\}$. Then $\langle L_1,\ldots,L_m\rangle \subset \Ann_{D(s,\nu)}(\phi)$.
\end{lemma}

\subsection{Differential operators from circuits and syzygies}\label{sec:diffcircuits}
Let $\ell_1,\ldots,\ell_m$ be the equations of $m$ hyperplanes through the origin in affine $n$-space. We are going to explain two strategies to compute annihilating differential operators from the combinatorics of the arrangement.
 
Let $k\leq m$. For fixed $\{i_1,\ldots,i_k\}\subset [m]$, we are going to denote
\begin{align*}
\partial_{\widehat{i_j}} \,=\, \partial_{c_{i_1}}\cdots \, \partial_{c_{i_j-1}}\cdot \partial_{c_{i_j+1}}\cdots \,\partial_{c_{i_k}} \, ,
\end{align*}
i.e., $\partial_{c_{i_j}}$ is left out.
Let $p_{i_1},\ldots,p_{i_k}, q \in \CC[c_1,\ldots,c_m]$ be polynomials such that  \begin{align*}
p_{i_1}(\ell_{i_1}-c_{i_1})+ \cdots +p_{i_k}(\ell_{i_k}-c_{i_k}) \,=\, q \,.
\end{align*}
Then the differential operator \begin{align}\label{eq:diffpq}
    p_{i_1}s_{i_1}{\partial_{\widehat{i}_1}} + \cdots +   p_{i_k}s_{i_k}{\partial_{\widehat{i}_k}} \,- \, q \, {\partial_{c_{i_1}} \cdots \, \partial_{c_{i_k}}} 
\end{align}
annihilates $\phi$, which one sees as follows. One substitutes $q = \sum_j p_{i_j} (\ell_{i_j} - c_{i_j})$, takes out the common factor $p_{i_j} \partial_{\widehat{i_j}}$,
and replaces multiplication by $(\ell_{i_j} - c_{i_j})$ with the action of~$\sigma_{s_{i_j}}$.

\medskip
We hence need to find possible candidates for $p_{i_1},\ldots,p_{i_k}$ and $q$, for which we present two strategies; one uses circuits, and the other one uses syzygies.
    
\paragraph{Circuits}
Collect the coefficients of the lines $\ell_i=a_1^{(i)}x_1+\cdots +a_n^{(i)}x_n$ as in \eqref{eq:elli}, $i=1,\ldots,m$, 
in the $n\times m$ matrix 
\begin{align}\label{eq:matrixA}
\mathcal{A} \,=\,
    \begin{bmatrix}
a^{(1)}_{1} & a^{(2)}_{1} & & \cdots & &  a^{(m)}_{1}\\
\vdots & \vdots & & & & \vdots \ \\
a^{(1)}_{n} & a^{(2)}_{n} &  & \cdots & & a^{(m)}_{n}\\
    \end{bmatrix} \, .
\end{align} Let $C=\{i_1,\ldots,i_k\}\subset [m]$ be a subset of dependent columns of~$\mathcal{A} $. Its corresponding submatrix of $\mathcal{A} $ is $\mathcal{A} _C \,=\, \left[
    a^{(i_1) }|a^{(i_2)}|\cdots | a^{(i_k)} \right]$.
We collect a basis of the kernel of $\mathcal{A}_C$ as columns of a matrix~$K_C$. 
We are going to denote it shorthand as $K_C=\ker(\mathcal{A} _C)$.
Every column of $K_C$ gives one possible choice for $p_{i_1},\ldots,p_{i_k}$, while $q$ is 
$[c_{i_1} \ c_{i_2} \ \cdots \  c_{i_k}]$
multiplied by this column. We are going to denote the set of resulting order-$k$ differential operators~\eqref{eq:diffpq} 
as~$\{P_j\}$. In the case that $C$ is a circuit, the resulting operator is unique up to a scalar, and we denote it by $P_C$.

Note that, for generic arrangements, it is sufficient to consider relations arising from circuits of~$\mathcal{A}$, since every non-minimal dependent subset of columns of $\mathcal{A} $ contains a circuit. Let $C=\{i_1,\ldots,i_k,i_{k+1}\}$ be a dependent set, such that every subset of size $k$ is a circuit, and denote a generator of the kernel of $\mathcal{A} _{C\setminus \{i_{k+1}\}}$ by~$K_{C\setminus \{i_{k+1}\}}$. Similarly, denote by $K_{C\setminus \{i_{1}\}}$ a generator of the kernel of $\mathcal{A} _{C\setminus \{i_{1}\}}.$ Since $\mathcal{A} _{C\setminus \{i_{k+1}\}}$ and $\mathcal{A} _{C\setminus \{i_{1}\}}$ are circuits, their kernels are one-dimensional. The kernel of $\mathcal{A}_{C}$ is generated by a vector of the form $(K_{C\setminus \{i_{k+1}\}},0)$ and a vector of the form $(0,K_{C\setminus \{i_{1}\}})$ (w.l.o.g., put $0$ in the first entry, otherwise reorder). Since they are linearly independent, they form a basis of the two-dimensional kernel $K_C.$ The operators constructed as in \eqref{eq:diffpq} then left-factor out $\partial_{i_{k+1}}$ and $\partial_{i_1}$, respectively, hence are left multiples of operators arising from circuits of~$\mathcal{A} $.

\paragraph{Syzygies} 
We now consider the case $q=0$ and are going to construct suitable $p_{i_j}$'s via a syzygy computation over the polynomial ring in the $x$- and $c$-variables.
We here start with the case $k=m$, i.e., $\{i_1,\ldots,i_k\}=[m]$.
To the coefficient matrix $\mathcal{A} $ as in~\eqref{eq:matrixA}, we attach the negative of the identity matrix $I_{k}$ below, and then compute syzygies of this matrix with the help of {\em Macaulay2}.
To be precise, we aim for syzygies in the $c$-variables, independent of the $x_i$'s. 
For computations, we set the degree of the $c$-variables to $0$, and that of the $x$-variables to~$1$, and then compute a basis of the degree-$0$ part of the syzygies.
%
A sample code is provided \Cref{example5lines}.
The output of that computation is a matrix $S=(S_1|\cdots| S_r)$ with entries in~$\CC[c_1,\ldots,c_m]$, where each column $S_i$ of $S$ gives a possible choice for $p_{i_1},\ldots,p_{i_k}$.
Then, for $j=1,\ldots,r$, the differential operator 
\begin{align}\label{eq:Qi}
   Q_j \,=\, (S_i)_1 s_{i_1} {\partial_{\widehat{i}_1}} + \cdots +   (S_i)_k s_{i_k} {\partial_{\widehat{i}_k}} 
\end{align}
 yields an operator of order $m-1$ in $\Ann_D(\phi)$. 
One can repeat the same strategy for any $\{i_1,\ldots ,i_k\} \subseteq [m]$ for any $k \leq m$. Note that this yields non-trivial differential operators only for $k$ sufficiently large.

\medskip
We summarize our findings of \Cref{sec:construction} in
\begin{theorem}\label{prop:ann}
{\mainprop}
\end{theorem}

We end this section by pointing out a possible connection to reciprocal linear spaces. To an $n \times m$ matrix $\mathcal{A}$, one associates the linear space
$ L_{\mathcal{A}} = \left\{ x^\top \mathcal{A} \, | \, x \in \CC^n \right\}$ in~$\CC^m $. The {\em reciprocal linear space} of $\mathcal{A}$ is the algebraic variety of entry-wise reciprocals of non-zero points of~$L_{\mathcal{A} }$,
\begin{align*}
 \mathcal{R}_{\mathcal{A}} \,=\,  
\overline{ \left\{\left( y_1^{-1},\ldots,y_m^{-1} \right) \, | \, y \in L_{\mathcal{A}} \cap \left( \CC\setminus \{0\} \right)^m \right\} } \,,
\end{align*}
where the Zariski closure in~$\CC^m$ is meant.
A {\em circuit} of $\mathcal{A}$ is a minimally dependent subset of columns of~$\mathcal{A}$. They give rise to a universal Gröbner basis of the defining ideal of~$\mathcal{R}_{\mathcal{A}}$  as follows.
\begin{theorem}[{\cite[Theorem 4]{circuit}}]\label{thm:recproc} Let $C$ be a circuit and collect a basis of its kernel as columns of a matrix~$K_C$. For the $i$-th column $(K_c^i)_{c \in C}$
of $K_C$, let  
\begin{align}\label{eq:groebnerRA}
f^i_C \,=\, \underset{c \,\in\, C}{\sum}K_c^i \underset{c' \in \, C \setminus \{c\}}{\prod}y_{c'} \, .
\end{align}
The set $\{f^i_C \ | \ C \text{ is a circuit of } \mathcal{A}\}$ is a universal Gröbner basis of the ideal of~$\mathcal{R}_{\mathcal{A}}$. 
\end{theorem}
The degree of the variety $\mathcal{R}_{\mathcal{A}}$ is the $\beta$-invariant of the matroid of~$\mathcal{A}$, and for $\mathcal{A}$ real, it equals the number of bounded regions enclosed by a generic displacement of the hyperplanes, cf.\ Varchenko's theorem~\cite[Theorem 1.2.1]{Varchenko}. We point out the structural resemblance of the syzygy operators $Q_j$~\eqref{eq:Qi} to the elements~\eqref{eq:groebnerRA} of the universal Gröbner basis.
We believe it would be worthwhile to understand this relation better and investigate if there is a relation that is similar to the close interplay of GKZ systems and toric varieties.

\subsection{Upper bound for the singular locus for lines in the plane}
Let $n=2$ and consider $m$ lines $\ell_i=a_ix+b_iy$, $i=1,\ldots,m$ such that $a_ib_j-a_jb_i\neq 0$ for any $i\neq j$. Denote by $I$ the $D_m(s,\nu)$-ideal generated by $H,\{L_i\},\{P_j\},$ and~$\{Q_k\}$. A set-theoretic upper bound of the singular locus of $I$ is given as follows. 
\begin{proposition}\label{prop:singlocn2}
For $m$ lines in the plane as above, the singular locus of the $D$-ideal $I=\langle H,\{L_i\},\{P_j\},\{Q_k\} \rangle$ is contained in the variety defined by the discriminantal arrangement of the line arrangement enhanced by the coordinate~axes.
\end{proposition}
\begin{proof}
Encode the $m$ lines in the matrix
\begin{align*}
\mathcal{A}_m \,=\,
    \begin{bmatrix}
a_{1} & a_{2} & & \cdots & &  a_{m}\\
b_1 & b_2 &  & \cdots & & b_m\\
    \end{bmatrix} \,  .
\end{align*}
The discriminantal arrangement is given by the maximal minors of the matrix \begin{align*}
\left[ \begin{array}{c|c}
\begin{matrix}
\text{Id}_2
\end{matrix} & \begin{matrix} \mathcal{A}_{m} \end{matrix} \\ \hline
\begin{matrix} 0 & 0 \end{matrix} & \begin{matrix} -c_1 & -c_2 & \hdots & -c_m \end{matrix} 
\end{array} \right] \, .
\end{align*}
Hence, the discriminantal arrangement consists of the factors 
\begin{enumerate}[(1)]
\item $c_i\,$ for $1 \leq i \leq m$, 
\item $a_ic_j-a_jc_i\,$ and $\,b_jc_i-b_ic_j\,$ for $1 \leq i < j \leq m$, and 
\item $(a_jb_k-a_kb_j)c_i+(a_kb_i-a_ib_k)c_j+(a_ib_j-a_jb_i)c_k\,$ for $1 \leq i<j<k \leq m.$
\end{enumerate}
They correspond to the minors of submatrices containing (1)~two columns, (2)~one column, and (3)~no column of the identity matrix. We are going to prove that each of them is contained in the saturated ideal $\operatorname{in}_{(0,1)}(I)\colon\langle \partial_{c_1},\ldots,\partial_{c_m}\rangle^\infty.$ Note that the initial ideal of $I$ w.r.t.\ $(0,1)\in \RR^{2m}$ is an ideal in the graded Weyl algebra, the polynomial ring $\operatorname{gr}_{(0,1)}(D_m(s,\nu))=\CC(s,\nu)[c_1,\ldots,c_m][ \partial_{c_1},\ldots,\partial_{c_m}]$.
Denote by $J$ the $ \operatorname{gr}_{(0,1)}(D_m(s,\nu))$-ideal 
generated by the initial forms of our operators, i.e., $J=\langle \operatorname{in}_{(0,1)}(H), \{ \operatorname{in}_{(0,1)}(L_i)\}, \{ \operatorname{in}_{(0,1)}(P_j)\}, \{ \operatorname{in}_{(0,1)}(Q_k)\} \rangle$. Clearly, one has the inclusion $J\subset \operatorname{in}_{(0,1)}(I)$.

For $\{i,j,k\}$ a circuit, the operator $P_{ijk}$ is
\begin{align*}
P_{ijk} & \,=\, ((a_jb_k-a_kb_j)c_i+(a_kb_i-a_ib_k)c_j  +(a_ib_j-a_jb_i)c_k)\partial_{c_i}\partial_{c_j}\partial_{c_k}  
\\
& \quad - (a_jb_k-a_kb_j)s_i\partial_{c_j}\partial_{c_k} 
+(a_kb_i-a_ib_k)s_j\partial_{c_i}\partial_{c_k}+(a_ib_j-a_jb_i)s_k\partial_{c_i}\partial_{c_j} \, , \ 
\end{align*}
and its initial form is 
$$\operatorname{in}_{(0,1)}(P_{ijk}) \,=\, ((a_jb_k-a_kb_j)c_i+(a_kb_i-a_ib_k)c_j+(a_ib_j-a_jb_i)c_k)\partial_{c_i}\partial_{c_j}\partial_{c_k} \,.$$ 
We have $\operatorname{in}_{(0,1)}(H)=c_1 \partial_{c_1}+ \cdots + c_m \partial_{c_m}$, and the initial form of $L_k$ is $$\operatorname{in}_{(0,1)}(L_k) \,=\, \big(\underset{1\leq i < j\leq m}{\sum} (a_ib_j+a_jb_i)c_k\partial_{c_i}\partial_{c_j}+ \underset{1\leq i\leq m}{\sum}a_ib_ic_k \partial_{c_i}^2\big)\partial_{c_k}\, .$$ 
In particular, we can directly see from the $L_i$'s and $P_j$'s that the minors of the first and third type, respectively, are contained in the singular locus. For factors of the second type, we will, w.l.o.g., consider $a_1c_2-a_2c_1$. We will show that is equivalent to $0$ in the quotient ring $\operatorname{gr}_{(0,1)}(D_m(s,\nu))/J,$ and hence contained in~$J$. Let us denote by $\partial$ the product $\partial_{c_1} \cdots \, \partial_{c_m}.$ 
In the quotient ring,
\begin{align*}
0 &\,\equiv\,\operatorname{in}(L_1) \,\equiv\, \operatorname{in}(L_1)\partial\,=\,  \big(\underset{1\leq i < j\leq m}{\sum} (a_ib_j+a_jb_i)c_1\partial_{c_i}\partial_{c_j}+ \underset{1\leq i\leq m}{\sum}a_ib_ic_1 \partial_{c_i}^2\big)\partial_{c_1} \partial\\
& \,=\, \big(\underset{1< i < j\leq m}{\sum} (a_ib_j+a_jb_i)c_1\partial_{c_i}\partial_{c_j}\partial_{c_1}+ \underset{1<j\leq m}{\sum} a_1b_jc_1\partial_{c_1}^2\partial_{c_j}\\
& \qquad + \underset{1<j \leq m}{\sum} a_jb_1c_1\partial_{c_1}^2\partial_{c_j}+
\underset{1< j\leq m}{\sum}a_jb_jc_1 \partial_{c_1} \partial_{c_j}^2+ a_1b_1c_1 \partial_{c_1}^3\big)\partial \,.
\end{align*} 
Using $\operatorname{in}_{(0,1)}(H)$, one obtains $a_1b_1c_1 \partial_{c_1}^3\equiv -\underset{1< j\leq m}{\sum}a_1b_1c_j \partial_{c_1}^2 \partial_{c_j},$ and similarly, we have $\underset{1<j\leq m}{\sum} a_1b_jc_1\partial_{c_1}^2\partial_{c_j}\equiv - \underset{1<j\leq m}{\sum} \ \underset{i \ne j}{\sum} a_1b_jc_i\partial_{c_1}\partial_{c_i}\partial_{c_j}-\underset{1<j\leq m}{\sum} a_1b_jc_j\partial_{c_1}\partial_{c_j}^2$.  
So, 
\begin{align*}
\operatorname{in}_{(0,1)}(L_k)&  \, \equiv \, (\underset{1< i < j\leq m}{\sum} (a_ib_j+a_jb_i)c_1\partial_{c_1}\partial_{c_i}\partial_{c_j}- \underset{1<j\leq m}{\sum} \ \underset{i \ne j}{\sum} a_1b_jc_i\partial_{c_1}\partial_{c_i}\partial_{c_j} + \\
& \qquad +\underset{1<j \leq m}{\sum} a_jb_1c_1\partial_{c_1}^2\partial_{c_j} -\underset{1<j\leq m}{\sum} a_1b_jc_j\partial_{c_1}\partial_{c_j}^2 +\\
& \qquad +\underset{1< j\leq m}{\sum}a_jb_jc_1 \partial_{c_1} \partial_{c_j}^2 -\underset{1< j\leq m}{\sum}a_1b_1c_j \partial_{c_1}^2 \partial_{c_j}) \partial \, .
\end{align*} 
Using the circuit operators, we can write $c_j$ in terms of $c_1$ and $c_i$ for any $i\geq 2$: $$c_j\,\equiv\, \frac{-(a_ib_j-a_jb_i)c_1-(b_1a_j-b_ja_1)c_i}{a_1b_i-a_ib_1} \, .$$
Substituting via that, one can show each sum term in $\operatorname{in}_{(0,1)}(L_k)$ factors out the term $(a_2c_1-a_1c_2)$, and that after factoring out, the remaining term is in the  $\CC(s,\nu)[\partial_{c_1},\ldots,\partial_{c_m}]$-ideal generated by $\partial_{c_1},\ldots,\partial_{c_m}$---independent of the $c$'s.
The statement now follows from the definition of the singular locus.
\end{proof}

\section{Examples}\label{sec:examples}
We here showcase our methods for point and line arrangements, including the two-site chain modeling a single exchange process in cosmology~\cite{DEcosmological}, and a hyperplane arrangement in $\CC^3$. We address bottlenecks of our construction so far, and make the dependency of our $D$-ideal on the hyperplane arrangement manifest. Unless stated otherwise, our computations in this section were carried in {\sc Singular}, with $\CC(s,\nu)$ as the field of coefficients, so that the results hold true for generic $(s,\nu)$.
We begin with points on a line.

\subsection{Points on the line}
Let $\mathcal{A}_{2}=\begin{bmatrix}
        1 & 2
\end{bmatrix}$ encode two copies of the origin on the complex line, so that we consider $\{x=c_1\}$ and $\{2x=c_2\}$. 
There is one circuit operator 
\begin{align*}
    P & \,=\, (-2c_1+c_2)\partial_1 \partial_2-s_2 \partial_1+2s_1 \partial_2\, ,
\end{align*}
the homogeneity operator $H$, and two operators constructed from the points:
\begin{align*}
    L_1 & \,=\, -\frac{1}{s_1(\nu_1-1)}c_1 \partial_{c_1}^2-\frac{2}{s_1(\nu_1-1)}c_1\partial_{c_1} \partial_{c_2}+\frac{s_1+\nu_1-1}{s_1(\nu_1-1)}\partial_{c_1}+\frac{2}{\nu_1-1} \partial_{c_2} \, , \\
    L_2 & \,=\, -\frac{1}{s_2(\nu_1-1)}c_2\partial_{c_1} \partial_{c_2}-\frac{2}{s_2(\nu_1-1)}c_2\partial_{c_2}^2+\frac{1}{\nu_1-1}\partial_{c_1}+\frac{2s_2+2\nu_1-2}{s_2(\nu_1-1)} \partial_{c_2} \, . 
\end{align*}
All together, these operators generate a holonomic $D$-ideal of holonomic rank~$2$ whose singular locus is $$V(c_1c_2(2c_1-c_2)) \, .$$

Let now $\mathcal{A}_{3}=\begin{bmatrix}
        1 & 2 &3
\end{bmatrix}$ be three points in $\mathbb{C}$.
There is one syzygy operator, five from the circuits, the homogeneity operator, and three corresponding to the points. The holonomic $D$-ideal generated by these operators has rank~$3$. Its singular locus is $$V(c_1c_2c_3(3c_2-2c_3)(2c_1-c_2)(3c_1-c_3))\, .$$

In general, let us consider $m$ points on a line, so that $\ell_i=a_ix$, $i=1,\ldots,m$.
The corresponding coefficient matrix then is $\mathcal{A}_m=[a_1 \ \cdots \ a_m].$ 

For $m>2,$ the circuit operators are 
\begin{align*}
P_{ij}  & \,=\, \, (s_ja_i \partial_{c_i}-s_ia_j \partial_{c_j})-(a_ic_j-a_jc_i)\partial_{c_i}\partial_{c_j},
\end{align*} 
where $i,j=1,\ldots,m$. For $m>3,$ the syzygy operators are
{\small
\begin{align*}
Q_{ijk} \,=\, \, (a_k^2c_j-a_ja_kc_k)s_i\partial_{c_j}\partial_{c_k}  +(-a_k^2c_i+a_ia_kc_k)s_j\partial_{c_i}\partial_{c_k} +(a_ja_kc_i-a_ia_kc_j)s_k \partial_{c_i}\partial_{c_j} \, ,
\end{align*} }

\noindent where $i,j,k=1,\ldots,m$. 
Finally, we have the following operators coming from the points:
\begin{align*} 
L_{i}  & \,=\, \, -\frac{a_i}{s_i} \partial_{c_i}+ \frac{c_i}{s_i(\nu_1-1)} \partial_{c_i}(a_1\partial_{c_1}+\cdots +a_m\partial_{c_m})-\frac{a_1}{\nu_1-1}\partial_{c_1}-\cdots-\frac{a_m}{\nu_1-1}\partial_{c_m} \, .
\end{align*} 
\begin{lemma}
Denote by $I$ the $D_m$-ideal generated by the $P_{ij}$'s,  $Q_{ijk}$'s, $L_i$'s, and the homogeneity operator $H$ from~\eqref{eq:homogeneous}. Then $\operatorname{rank}(I)\leq m$.
\end{lemma}
\begin{proof}
We aim to determine the dimension of $R_m/R_mI$ as a $\CC(c)$-vector space. First, we use the operator $H$ to write $[\partial_{c_m}]\in R_m/R_mI$ as a $\CC(c)$-linear combination of $\partial_{c_1},\ldots,\partial_{c_{m-1}}$, and $1$. Secondly, for $i \ne j<m$, we use the circuit operators to write $\partial_{c_i}\partial_{c_j}$ as a linear combination of $\partial_{c_i}$ and $\partial_{c_j}.$ Finally, for $i<m$, we exploit the 
operators $L_1,\ldots,L_m$ to write $\partial_{c_i}^2$ as a linear combination of $1,\partial_{c_1},\ldots,\partial_{c_{m}}$ and $\partial_{c_i}\partial_{c_j}$ for every pair $i \ne j \leq m$. Hence, $1, \partial_{c_1},\ldots,\partial_{c_{m-1}}$ also generate every $\partial_{c_i}^2$. 
\end{proof}

The factors of the discriminantal arrangement are $(a_jc_i-a_ic_j)_{i<j}$ and $c_i$, since we extend the arrangement by the coordinate hyperplanes (here, the origin). We have the following relation to the singular locus of our $D$-ideal.
\begin{lemma}
The singular locus of the $D_m$-ideal $I$ is contained in the discriminantal arrangement of $\mathcal{A} $ enhanced by the origin. 
\end{lemma}
\begin{proof}
By taking the initial form of the operators $P_{ij}$ and $L_i$ w.r.t.\ the weight vector $(0,1)\in \RR^{2m}$, respectively, one reads that $(a_ic_j-a_jc_i)\partial_{c_i}\partial_{c_j} \in \operatorname{in}_{(0,1)}(I)$ and $\frac{c_i}{s_i(\nu_1-1)} \partial_{c_i}(a_1\partial_{c_1}+\cdots +a_m\partial_{c_m}) \in \operatorname{in}_{(0,1)}(I)$. By eliminating the $\partial_{c_i}$'s, one reads that the singular locus of $I$ is contained in
$V\big(\prod_{i<j}\left( a_ic_j-a_jc_i\right)\cdot \prod_{i=1}^m c_i \big)$.
\end{proof}

\subsection{Moving two lines in the plane}
We here consider the lines $\ell_1(x,y)=3x+5y$ and $\ell_2(x,y)=7x-3y$ in the plane. 
We do not have any operators arising from syzygies or circuits here. To the two lines, we associate the shift operators
\begin{align*}
3\sigma_{\nu_2}^{-1}\sigma_{s_1}^{-1}+5\sigma_{\nu_1}^{-1}\sigma_{s_1}^{-1}-c_1\sigma_{\nu_1}^{-1}\sigma_{\nu_2}^{-1}\sigma_{s_1}^{-1}-\sigma_{\nu_1}^{-1}\sigma_{\nu_2}^{-1} \, ,\\
7\sigma_{\nu_2}^{-1}\sigma_{s_2}^{-1}-3\sigma_{\nu_1}^{-1}\sigma_{s_2}^{-1}-c_2\sigma_{\nu_1}^{-1}\sigma_{\nu_2}^{-1}\sigma_{s_2}^{-1}-\sigma_{\nu_1}^{-1}\sigma_{\nu_2}^{-1} \, ,
\end{align*}
from which we deduce the differential operators $L_1,L_2$ via \eqref{eq:shiftsdiffsc} and \eqref{eq:shiftnu}. For generic $s_1,s_2,\nu_1,\nu_2$, they generate a holonomic $D$-ideal of holonomic rank $7$. Extending the $D$-ideal by $H$ from~\eqref{eq:homogeneous} yields holonomic rank $3$, the number of bounded chambers. Its singular locus is $V(c_1c_2(3c_1 + 5c_2)(7c_1 - 3c_2))$, which coincides precisely with the discriminantal arrangement.

\subsection{Lines parallel to the coordinate axes}\label{ex:linesparall}
Consider the two lines $\ell_1=x$ and $\ell_2=y$ in the plane. Shifted by $c_i$ each, they enclose one bounded region with the coordinate axes---a quadrilateral. Again, we do not have any non-trivial syzygies or circuits here, so that our only operators arise from the lines themselves. They are
\begin{align*}
   L_1  \, = \, -\frac{c_1}{s_1(\nu_1-1)}\partial_{c_1}^2+\frac{s_1+\nu_1-1}{s_1(\nu_1-1)}\partial_{c_1} \, , \ \ 
L_2  \, = \, -\frac{c_2}{s_2(\nu_2-1)}\partial_{c_2}^2+\frac{s_2+\nu_2-1}{s_2(\nu_2-1)}\partial_{c_2} \, ,
\end{align*}
and together, they generate a holonomic $D_2$-ideal of holonomic rank~$4$. Extending the $D$-ideal by $H$~\eqref{eq:homogeneous} yields holonomic rank~$1$.
The singular locus of $I=\langle H,L_1,L_2\rangle$ is the union of the coordinate axes, which coincides with the discriminantal arrangement.
Its associated $D$-module $D(s,\nu)/I$ recovers the direct image $D$-module of $D(s,\nu)/J$ under the projection 
to the $c$-coordinates, where $J$ denotes the $D$-ideal generated by the four operators
\begin{align*}
& (c_1-x)x\partial_x+xs_1+(x-c_1)(\nu_1-1) \, , \ \
(c_2-y)y\partial_y+ys_2+(y-c_2)(\nu_2-1)\, , \\
& \qquad \qquad \qquad \quad \ \ \,(c_1-x)\partial_{c_1}-s_1 \, , \ \   (c_2 - y)\partial_{c_2} - s_2 \, , 
\end{align*}
which annihilate the integrand of the correlator. For the definition of direct images of $D$-modules, see~\cite[Section~1.5]{HTT08}.

\subsection{Moving three lines in the plane}\label{sec:threelines}
Consider the lines encoded by the columns of $\mathcal{A}_3=\left[\begin{smallmatrix}
3 & 7 & 1 \\
5 & -3 & -2 \\
    \end{smallmatrix} \right]$.
%
We obtain one circuit operator
\begin{align*}
P \,=\, s_1 \partial_{c_2}\partial_{c_3} - s_2\partial_{c_1}\partial_{c_2}+4s_3\partial_{c_1}\partial_{c_2}  -(c_1-c_2+4c_3)\partial_{c_1}\partial_{c_2}\partial_{c_3} \, ,
\end{align*}
no syzygy operator, and one operator $L_i$ for each of the three lines. Together with the operator $H$ from~\eqref{eq:homogeneous}, they generate a holonomic $D_3$-ideal of holonomic rank~$6$. Again, the singular locus of this \mbox{$D_3$-ideal} coincides with the discriminantal arrangement of $\mathcal{A}_3$ enhanced by the coordinate axes.
Its factors are
\begin{align*}
    c_1,\ c_2,\ c_3,\ 3c_1+5c_2,\ 7c_1-3c_2,\ 2c_1+5c_3 \, ,\\ c_1-3c_3,\ 2c_2-3c_3, \ c_2-7c_3, \ c_1-c_2+4c_3 \, .
\end{align*}

\subsection{An example from cosmology}\label{sec:cosmology}
The cosmological two-site chain encodes a single exchange process and gives rise to three lines in the plane, namely $\ell_1=x+y$, $\ell_2=x$, and $\ell_3=y$, cf.~\cite[(1.7)]{DEcosmological}. 
The correlator function~is 
\begin{align*}
\phi(c)  \,=\, \int_\Gamma (\ell_1-c_1)^{s_1}(\ell_2-c_2)^{s_2}(\ell_3-c_3)^{s_3}x^{\nu_1}y^{\nu_2} \, \frac{\d x}{x} \, \frac{\d y}{y}\, .
\end{align*}
The only differential operator arising from the circuits of $\mathcal{A}_{2\text{-site}}=\left[\begin{smallmatrix}
        1 & 1 & 0\\
        1 & 0 & 1
    \end{smallmatrix}\right]$ is 
    \begin{align*}
    P \,=\, (c_1 - c_2 - c_3) \partial_{c_1} \partial_{c_2} \partial_{c_3} + (-s_1 \partial_{c_2} \partial_{c_3}  + s_2 \partial_{c_1} \partial_{c_3}  +s_3 \partial_{c_1} \partial_{c_2} ) \,.
    \end{align*}

\noindent From the lines themselves, we read three differential operators of order $3$ in~$\Ann(\phi)$:
 \begin{align*}
   L_1 & \,=\ -\frac{1}{s_1(\nu_1-1)(\nu_2-1)}c_1\partial_{c_1}^3-\frac{1}{s_1(\nu_1-1)(\nu_2-1)}c_1\partial_{c_1}^2\partial_{c_2}  \\
   & \quad -\frac{1}{s_1(\nu_1-1)(\nu_2-1)}c_1\partial_{c_1}^2\partial_{c_3} -\frac{1}{s_1(\nu_1-1)(\nu_2-1)}c_1\partial_{c_1}\partial_{c_2}\partial_{c_3} +\\
   & \quad +\frac{s_1+\nu_1+\nu_2-2}{s_1(\nu_1-1)(\nu_2-1)}\partial_{c_1}^2+\frac{s_1+\nu_2-1}{s_1(\nu_1-1)(\nu_2-1)}\partial_{c_1}\partial_{c_2}+ \\
   &\quad + \frac{s_1+\nu_1-1}{s_1(\nu_1-1)(\nu_2-1)}\partial_{c_1}\partial_{c_3}+\frac{1}{(\nu_1-1)(\nu_2-1)}\partial_{c_2}\partial_{c_3} \, , \\
   L_2 & \,=\, -\frac{1}{s_2(\nu_1-1)}c_2\partial_{c_1}\partial_{c_2}-\frac{1}{s_2(\nu_1-1)}c_2\partial_{c_2}^2+\frac{1}{\nu_1-1}\partial_{c_1}+\frac{s_2+\nu_1-1}{s_2(\nu_1-1)}\partial_{c_2} \, ,\\
   L_3 & \,=\, -\frac{1}{s_3(\nu_2-1)}c_3\partial_{c_1}\partial_{c_3}-\frac{1}{s_3(\nu_2-1)}c_3\partial_{c_3}^2+\frac{1}{\nu_2-1}\partial_{c_1}+\frac{s_3+\nu_2-1}{s_3(\nu_2-1)}\partial_{c_3} \, .
 \end{align*}
Together with the circuit operator and the operator~$H$, they generate a holonomic $D_3$-ideal of holonomic rank~$4$ whose singular locus is
\begin{align*}
V\left(c_1c_2c_3(c_1 - c_2)(c_1 - c_3)(c_1 - c_2 - c_3)\right) \, \subset \, \CC^3 \,.
\end{align*}
It coincides exactly with the discriminantal arrangement of the central arrangement $\{\ell_1,\ell_2,\ell_3\}$ enhanced by the coordinate axes. 

It turns out that in this example, our combinatorially constructed $D$-ideal coincides with a certain integration $D$-ideal, see \cite[(5.8)]{SST00} for the definition. For that, observe that the following operators annihilate the integrand of~$\phi$:
 \begin{align*}
(\ell_1-c_1)(\ell_2-c_2)x\partial_x-s_1(\ell_2-c_2)x-s_2(\ell_1-c_1)x-(\nu_1-1)(\ell_1-c_1)(\ell_2-c_2) \, , \\ (\ell_1-c_1)(\ell_3-c_3)y\partial_y-s_1(\ell_3-c_3)y-s_3(\ell_1-c_1)y-(\nu_2-1)(\ell_1-c_1)(\ell_3-c_3)\, ,\\ (\ell_1-c_1)\partial_{c_1}+s_1 \, ,\ (\ell_2-c_3)\partial_{c_2}+s_2\, ,\ (\ell_3-c_3)\partial_{c_3}+s_3 \, . \  \qquad \qquad \quad
\end{align*}
Denote by $J$ the $D$-ideal generated by them, and denote by $\pi_c$ the projection onto $\CC^3$ in the $c$-variables.\
Using {\em Macaulay2}, one computationally confirms that for random $s$ and $\nu$'s, the direct image $D$-module ${\pi_c}_+ (D/J)$ is concentrated in degree $0$ and is of the form $D/N$, with $N$ the $D$-ideal generated by
\begin{align*}
    N_1 & \,=\, c_1\partial_{c_1}+c_2\partial_{c_2}+c_3\partial_{c_3}-
    (\nu_1+\nu_2+s_1+s_2+s_3)\, , \\
 N_2 & \,=\,   c_3\partial_{c_1}\partial_{c_3} + c_3\partial_{c_3}^2 -s_3\partial_{c_1}-(s_3+\nu_2-1)\partial_{c_3} \, , \\
 N_3 & \,=\,  c_2\partial_{c_1}\partial_{c_2} + c_2\partial_{c_2}^2 -s_2\partial_{c_1} -(s_2+\nu_1-1)\partial_{c_2}  \, .
 \end{align*}
The holonomic rank of $N$ is $4$, and it coincides with our $D$-ideal.

In the cosmological setup, one sets $s_1=s_2=s_3=-1$, and $\nu_1=\nu_2\eqqcolon \varepsilon$ is related to the expansion rate of the universe. 
Explicitly, our operators then become
\begin{align*}
H & \,=\, c_1\partial_{c_1}+c_2\partial_{c_2}+c_3\partial_{c_3}-(2\varepsilon - 3) \, , \\
    P & \,=\, (c_1 - c_2 - c_3) \partial_{c_1} \partial_{c_2} \partial_{c_3} + ( \partial_{c_2} \partial_{c_3}  -\partial_{c_1} \partial_{c_3} -\partial_{c_1} \partial_{c_2} ) \, ,\\
   L_1 & \,=\ \frac{1}{(\varepsilon-1)^2} \cdot \big[ c_1\partial_{c_1}^3+c_1\partial_{c_1}^2\partial_{c_2}   +c_1\partial_{c_1}^2\partial_{c_3} +c_1\partial_{c_1}\partial_{c_2}\partial_{c_3} \\
   & \quad -{(2\varepsilon-3)}\partial_{c_1}^2-(\varepsilon-2)\partial_{c_1}\partial_{c_2} - (\varepsilon-2)\partial_{c_1}\partial_{c_3}+\partial_{c_2}\partial_{c_3} \big] \, , \\
   L_2 & \,=\, \frac{1}{\varepsilon-1} \left( c_2\partial_{c_1}\partial_{c_2}+c_2\partial_{c_2}^2+\partial_{c_1}-(\varepsilon-2)\partial_{c_2}\right) \, ,\\
   L_3 & \,=\, \frac{1}{\varepsilon-1}\left( c_3\partial_{c_1}\partial_{c_3}+c_3\partial_{c_3}^2+\partial_{c_1}-(\varepsilon-2)\partial_{c_3} \right)\, ,
 \end{align*}
and (after substituting $\varepsilon\to \varepsilon+1$, to~match~notation) recover the restricted GKZ system for the cosmological correlator of the two-site chain as computed~in~\cite{DEcosmological}.
\begin{remark}[Double exchange process] 
For the three-site chain, one obtains six hyperplanes $\ell_1,\ldots,\ell_6$ in $\RR^3$, encoded by 
\begin{align*}
    \mathcal{A}_{3\text{-site}} \,=\, \begin{bmatrix}
         1 & 1 & 0 & 1 & 0 & 0\\
         1 & 1& 1 & 0 & 1 & 0\\
         1 & 0 &1 &0 &0  &1
    \end{bmatrix} \, ,
\end{align*}
see \cite[(2.29)]{DEcosmological}. The GKZ system of the associated generalized Euler integral $\int \prod_{i=1}^6 (\ell_i-c_i)x^\nu \frac{\d x}{x} $, when leaving the coefficients of all $\ell_i$ generic, has holonomic rank $30$. The $D$-ideal $\langle H,\{L_i\},\{P_C\}_{C\text{ a circuit}},\{Q_k\}\rangle$, 
for fixed coefficients determined by $\mathcal{A}_{3\text{-site}}$, has holonomic rank $30$ as well, which we computed in {\em Macaulay2} for randomly chosen values of $\nu$ and $s$. 
The integrand of the cosmological correlator function is $\prod_{i\neq 5} (\ell_i-c_i)^{-1}+\prod_{i\neq 6}(\ell_i-c_i)^{-1}$. Using our methods, one could, in principle, compute annihilating $D_5$-ideals for each of the summands separately, and then compute an annihilating $D_5$-ideal for their~sum.
\end{remark}

\subsection{Different representations of a uniform matroid}\label{sec:U23}
The example in this section shows that our $D$-ideal does not depend only on the matroid, but on the hyperplane arrangement itself.
Consider the two matrices
\begin{align*}
    \mathcal{A}_{2\text{-site}}\,=\, \begin{bmatrix}
        1 & 1 & 0\\
        1 & 0 & 1
    \end{bmatrix} \quad \text{and} \quad
      \mathcal{B}\,=\, \begin{bmatrix}
        1 & 1 & 0\\
        -1 & 0 &  1 
    \end{bmatrix}\, .
\end{align*}
They give rise to the same matroid, namely the uniform matroid $U_{2,3}$ of rank $2$ on $3$ elements. To relate the resulting $D_3$-ideals, we are going to use that
\begin{align}\label{eq:BA}
    \mathcal{B} \,=\, \operatorname{diag}(-1,1)\cdot \mathcal{A}_{2\text{-site}} \cdot \operatorname{diag}(-1,-1,1) \, .
\end{align}
The kernel of $\mathcal{B} $ is spanned by $[-1,1,-1]^\top$, which gives rise to the operator 
\begin{align*}
   P \,=\, (c_1-c_2+c_3)\partial_1\partial_2\partial_3+\left( -s_1\partial _{c_2}\partial_{c_3}+s_2\partial_{c_1}\partial_{c_3} - s_3 \partial_{c_1}\partial_{c_3} \right) \,.
\end{align*}
The three lines encoded by $\mathcal{B} $, i.e., $\ell_1=x$, $\ell_2=y$, $\ell_3=x-y$, induce the operators $L_1,L_2,L_3$.
We can pass from the operators for $\mathcal{A}_{2\text{-site}}$ to the operators for $\mathcal{B} $ as follows. First, for the operators coming from the lines, we can use the equality~\eqref{eq:BA}. This operation will transform the lines of $\mathcal{A}_{2\text{-site}}$ into the lines of $\mathcal{B} $, and hence the operators as well.
For the circuit, if we multiply the kernel $[1 -1 -1]^\top$ of $\mathcal{A}_{2\text{-site}}$ by the $3 \times 3$ matrix in~\eqref{eq:BA}, we get the kernel $[-1 \ 1 -1]^\top$ of $\mathcal{B} $, so we can deduce the circuit operators of $\mathcal{B} $ from the operators of $\mathcal{A}_{2\text{-site}}$.

Together, the operators $H,L_1,L_2,L_3$, and $P$ generate a holonomic $D_3$-ideal of holonomic rank~$4.$ Its singular locus is
\begin{align*}
    V\left( c_1c_2c_3(c_1 - c_2)(c_1 + c_3)(c_1 - c_2 + c_3) \right) \, \subset \, \CC^3 \, .
\end{align*}
Its factors again coincide precisely with the discriminantal arrangement of $\mathcal{B}$. 

\subsection{Moving five lines in the plane}\label{example5lines}
We revisit the three lines from \Cref{sec:threelines} and add two more lines via
\begin{align*}
\mathcal{A}_5 \,=\, 
    \begin{bmatrix}
3 & 7 & 1 & 1 & 3\\
5 & -3 & -2 & -1 & 1\\
    \end{bmatrix} \, .
\end{align*}
Any generic displacement of this line arrangement encloses $15$ bounded regions together with the coordinate axes.
The kernel of $\mathcal{A}_5$ is spanned by the~columns~of
\begin{align*}
  K \,= \,  
  \begin{bmatrix}
1 & 0 & 0\\
0 & 0 & 1\\
-12 & 4 & 16\\
24 & -7 & -32\\
-5 & 1 & 3\\
    \end{bmatrix} \, .
\end{align*} 
So, the operators
\begin{align*} \begin{split}
P_0  & \,=\,(-c_1 + 12c_3 - 24c_4 + 5c_5) {\partial_{c_{1}}\partial_{c_{2}}\partial_{c_{3}}\partial_{c_{4}}\partial_{c_{5}}}+\\
& \quad +(s_1 {\partial_{c_{2}}\partial_{c_{3}}\partial_{c_{4}}\partial_{c_{5}}} - 12s_3 {\partial_{c_{1}}\partial_{c_{2}}\partial_{c_{4}}\partial_{c_{5}}} + 24s_4 {\partial_{c_{1}}\partial_{c_{2}}\partial_{c_{3}}\partial_{c_{5}}} - 5s_5 {\partial_{c_{1}}\partial_{c_{2}}\partial_{c_{3}}\partial_{c_{4}}}) \, , \\
P_1 & \,=\, (- 4c_3 + 7c_4 - 1c_5) {\partial_{c_{1}}\partial_{c_{2}}\partial_{c_{3}}\partial_{c_{4}}\partial_{c_{5}}}+ \\
& \quad +(4s_3 {\partial_{c_{1}}\partial_{c_{2}}\partial_{c_{4}}\partial_{c_{5}}} -7s_4 {\partial_{c_{1}}\partial_{c_{2}}\partial_{c_{3}}\partial_{c_{5}}} + s_5 {\partial_{c_{1}}\partial_{c_{2}}\partial_{c_{3}}\partial_{c_{4}}}) \, ,\\
P_2  & \,=\, (- c_2 - 16c_3 + 32c_4 - 3c_5) {\partial_{c_{1}}\partial_{c_{2}}\partial_{c_{3}}\partial_{c_{4}}\partial_{c_{5}}} +\\
& \quad +(s_2 {\partial_{c_{1}}\partial_{c_{3}}\partial_{c_{4}}\partial_{c_{5}}} + 16s_3 {\partial_{c_{1}}\partial_{c_{2}}\partial_{c_{4}}\partial_{c_{5}}} -32s_4 {\partial_{c_{1}}\partial_{c_{2}}\partial_{c_{3}}\partial_{c_{5}}} + 3s_5 {\partial_{c_{1}}\partial_{c_{2}}\partial_{c_{3}}\partial_{c_{4}}}) \,
\end{split}\end{align*} 
annihilate $\phi$.
For computing the syzygies, we run the following code in {\em Macaulay2}: 
{\footnotesize 
\begin{verbatim}
k = 5; R = QQ[c_1..c_k, Degrees => {k:0}]; S = R[x,y]
A = matrix{{3,7,1,1,3},{5,-3,-2,-1,1}}
B = (-id_(QQ^k)) || A
L = (vars R | vars S) * B
basis(0, syz L) -- find solutions only in the c's
\end{verbatim}
} 
\noindent This provides a matrix 
\begin{align*}
S \,=\, 
\begin{bmatrix}
 0 & 4c_3-7c_4+c_5 & c_2-4c_4-c_5 \\
4c_3-7c_4+c_5 & 0 & -c_1-3c_4+2c_5 \\
-4c_2+16c_4+4c_5 & -4c_1-12c_4+8c_5 & 0 \\
7c_2-16c_3-11c_5 & 7c_1+12c_3-11c_5 & 4c_1+3c_2-11c_5 \\
-c_2-4c_3+11c_4 & -c_1-8c_3+11c_4 & c_1-2c_2+11c_4
\end{bmatrix}\, ,
\end{align*}  
from which we derive the following operators in $\Ann_{D_5}(\phi)$:
\begin{align*}
\begin{split}
Q_0  & \,=\,
((4c_3-7c_4+c_5)s_2 {\partial_{c_{1}}\partial_{c_{3}}\partial_{c_{4}}\partial_{c_{5}}} +(-4c_2+16c_4+4c_5)s_3 {\partial_{c_{1}}\partial_{c_{2}}\partial_{c_{4}}\partial_{c_{5}}}+\\
& \qquad + (7c_2-16c_3-11c_5)s_4 {\partial_{c_{1}}\partial_{c_{2}}\partial_{c_{3}}\partial_{c_{5}}} +(-c_2-4c_3+11c_4)s_5 {\partial_{c_{1}}\partial_{c_{2}}\partial_{c_{3}}\partial_{c_{4}}}) \, , \\
Q_1 & \,=\,
((4c_3-7c_4+c_5)s_1 \partial_{c_{2}}\partial_{c_{3}}\partial_{c_{4}}\partial_{c_{5}} +(-4c_1-12c_4+8c_5)s_3 \partial_{c_{1}}\partial_{c_{2}}\partial_{c_{4}}\partial_{c_{5}}+\\
& \qquad + (7c_1+12c_3-11c_5)s_4 \partial_{c_{1}}\partial_{c_{2}}\partial_{c_{3}}\partial_{c_{5}} +(-c_1-8c_3+11c_4)s_5 \partial_{c_{1}}\partial_{c_{2}}\partial_{c_{3}}\partial_{c_{4}}) \, , \\
Q_2 & \,=\,
((c_2-4c_4-c_5)s_1 \partial_{c_{2}}\partial_{c_{3}}\partial_{c_{4}}\partial_{c_{5}}+(-c_1-3c_4+2c_5)s_2 \partial_{c_{1}}\partial_{c_{3}}\partial_{c_{4}}\partial_{c_{5}} +\\
& \qquad + (4c_1+3c_2-11c_5)s_4 \partial_{c_{1}}\partial_{c_{2}}\partial_{c_{3}}\partial_{c_{5}} +(c_1-2c_2+11c_4)s_5 \partial_{c_{1}}\partial_{c_{2}}\partial_{c_{3}}\partial_{c_{4}}) \,.
\end{split} 
\end{align*}
We repeat the operations for any $\{i_1,...,i_k\} \subseteq [5]$ for $k=4,5$. For, $k=3$, \linebreak {\em Macaulay2} returns only the zero vector; hence, this does not contribute to any operator.
There are also the operators $L_1,\ldots,L_5$ constructed from the lines and the homogeneity operator~$H$.
We computed for several randomly chosen values of $s,\nu$ that, all together, they generate a $D_5$-ideal of holonomic rank~$15$. 
The computation of the singular locus did not terminate. 

\begin{remark}
We observed in all of our examples that the two $D$-ideals\linebreak
$\langle H,\{L_i\},\{P_j\},\{Q_k\}\rangle$ and $\langle H,\{L_i\},\{P_j\}\rangle$ coincide.
\end{remark}

\section{Outlook}\label{sec:outlook}
In this article, we tackled the combinatorial encryption of Mellin integrals of individual powers of hyperplanes as holonomic functions in the constant terms of the hyperplanes. 
Our presentation so far is a case study from which several interesting follow-up questions and paths for future research arise. We provided a combinatorial construction of an annihilating $D$-ideal for the combinatorial correlator function. Our examples of line arrangements in the plane suggest that they might compute restrictions of GKZ systems. 

Since our annihilating $D$-ideal does not depend on the matroid of the hyperplane arrangement only, one should check if and how our combinatorially obtained $D$-ideal can be extended by utilizing the logarithmic derivation module of the arrangement. We also plan to investigate to what extent the observed interplay of the singular locus of our \mbox{$D$-ideal} and the discriminantal arrangement of the central arrangement, and that of the holonomic rank of our \mbox{$D$-ideal} and the $\beta$-invariant of the matroid hold true. 

Functions of a highly similar structure are Igusa zeta functions, for which one has evaluation formulae in the $p$-adic case~\cite{MaglioneVoll}. It would be worthwhile to check if these formulae can be generalized beyond the $p$-adic~case.

\bigskip 

\noindent {\bf Acknowledgements.}
We thank Johannes Henn, Martina Juhnke, Lukas Kühne, Saiei-Jaeyeong Matsubara-Heo, Mahrud Sayrafi,
Bernd Sturmfels, Simon Telen, Francisco Vazão, and Christopher Voll for insightful discussions, and Felix Lotter, Guilherme L.\ Pimentel, and Dawei Shen for their helpful comments on an early version of our manuscript.

\medskip
\noindent {\bf Funding statement.}
The research of ALS is funded by the European Union (ERC, UNIVERSE PLUS, 101118787). Views and opinions expressed are, however, those of the author(s) only and do not necessarily reflect those of the European Union or the European Research Council Executive Agency. Neither the European Union nor the granting authority can be held responsible for them.

\vspace*{-6mm}

\begin{small}

\end{small}

\end{document}